\documentclass[12pt]{article}
\usepackage{amsmath,amsthm, amssymb}
\usepackage{amssymb,latexsym}
\usepackage[hidelinks]{hyperref}
\usepackage{caption}
\newtheorem{theorem}{Theorem}
\newtheorem{conjecture}{Conjecture}
\newtheorem{lemma}{Lemma}
\newtheorem{question}{Question}
\newtheorem{remark}{Remark}

\newcommand{\seqnum}[1]{\href{https://oeis.org/#1}{\rm \underline{#1}}}

\begin{document}

\baselineskip=17pt

\title{\bf Generalizations of amicable numbers}

\author{\bf S. I. Dimitrov}

\date{}

\maketitle
\begin{abstract}
In this paper, we propose new generalizations of amicable numbers. We also give examples and prove properties of these new concepts.\\
\quad\\
\textbf{Keywords}:  Sum of divisors, Amicable numbers.\\
\quad\\
{\bf  2020 Math.\ Subject Classification}: 11A25 $\cdot$ 11D72 
\end{abstract}

\section{Notations}
\indent

One of the most remarkable sums in number theory is the aliquot sum $s(n)$ of a positive integer $n$. 
It is the sum of all proper divisors of $n$, that is, all divisors of $n$ other than $n$ itself. Thus 
\begin{equation*}
s(n)=\sum _{{d|n,} \atop {d<n}}d\,.
\end{equation*}
As usual $\sigma(n)$ is the sum of all the divisors of $n$. Thus $\sigma(n)=s(n)+n$. 
We let $\zeta(s)$ denote the Riemann zeta function. Recall that
\begin{equation*}
\zeta(s)=\sum\limits_{n=1}^\infty\frac{1}{n^s}
\end{equation*}
for $\textmd{Re}(s)>1$. The abbreviations $\textmd{gcd}$ and $\textmd{lcm}$ stand for the greatest common divisor and the least common multiple.

\section{Introduction}
\indent

Two natural numbers $m$ and $n$ are said to be amicable if
\begin{equation*}
\sigma(m)=\sigma(n)=m+n\,.
\end{equation*}
If $m=n$, then $m$ is perfect. It is well known that the smallest perfect number is 6.
The smallest pair of amicable numbers, (220, 284), was known to the Pythagoreans.
Although more than 2000 years have passed, it is still unknown whether infinitely many amicable pairs exist.
In 1913 Dickson \cite{Dickson} defined the natural numbers $n_1,n_2, n_3$ to form an amicable triple if
\begin{equation*}
\sigma(n_1)=\sigma(n_2)=\sigma(n_3)=n_1+n_2++n_3
\end{equation*}
or equivalently 
\begin{equation*}
\left|\begin{array}{ccc}
s(n_1)=n_2+n_3\\
s(n_2)=n_1+n_3\\
s(n_3)=n_1+n_2
\end{array}\right.\,.
\end{equation*}
The smallest amicable triple, in the sense defined by Dickson, is (1980, 2016, 2556). The same generalization to the $k$-tuples was also proposed by Dickson. 
Afterwards, Yanney \cite{Yanney}, using the same symbolism and terminology, defined the numbers $n_1,n_2, n_3$ to form an amicable triple if
\begin{equation*}
2\sigma(n_1)=2\sigma(n_2)=2\sigma(n_3)=n_1+n_2++n_3
\end{equation*}
or equivalently 
\begin{equation*}
\left|\begin{array}{ccc}
n_1=s(n_2)+s(n_3)\\
n_2=s(n_1)+s(n_3)\\
n_3=s(n_1)+s(n_2)
\end{array}\right.\,.
\end{equation*}
The smallest amicable triple according to Yanney's definition is (238, 255, 371). The same generalization to the $k$-tuples was also proposed by Yanney. 
Carmichael \cite{Carmichael} defined a multiply amicable pair as two numbers $m$ and $n$ such that
\begin{equation*}
\sigma(m)=\sigma(n)=t(m+n)\,,
\end{equation*}
where $t$ is a positive integer. The same generalization to the $k$-tuples was proposed by Mason \cite{Mason}. 
In 1995 Cohen, Gretton and Hagis \cite{Cohen} defined natural numbers $m$ and $n$ to be multiamicable if the sum of the
proper divisors of each is a multiple of the other. More precisely, the pair $(m, n)$ is said to be $(\alpha, \beta)$-amicable if
\begin{equation*}
\sigma(m)-m=\alpha n \quad \mbox{ and }   \quad \sigma(n)-n=\beta m\,.
\end{equation*}
Another interesting generalization is due to Bishop, Bozarth, Kuss and Peet \cite{Bozarth}. 
They defined a feebly amicable $k$-tuple as numbers $n_1,\ldots, n_k$ such that
\begin{equation*}
\frac{n_1}{\sigma(n_1)}+\cdots+\frac{n_k}{\sigma(n_k)}=1\,.
\end{equation*}
Apparently, every amicable $k$-tuple satisfies the equation above. For $k=2$, the smallest feebly amicable pair is (4, 12).
In Section \ref{WHM}, we will generalize the definition given in \cite{Bozarth}.
We say that an integer $n$ is an amicable number if it belongs to an amicable pair, or equivalently 
\begin{equation*}
\sigma(\sigma(n)-n)=\sigma(n)\,.
\end{equation*}
Let $A(x)$ denote the number of amicable numbers up to $x$.
In 1955, Erdős \cite{Erdos1} proved that the asymptotic density of amicable integers relative to the positive integers is 0. 
That is, the ratio of the number of amicable integers less than $x$ to $x$ tends to zero as $x$ tends to infinity.
Subsequently, several papers \cite{Erdos2, Pomerance1, Pomerance2, Rieger} established progressively better upper bounds for $A(x)$ 
and the best result to date is due to Pomerance \cite{Pomerance3} with
\begin{equation*}
A(x)\leq\frac{x}{e^{\sqrt{\log x}}}\,.
\end{equation*}
A lot of articles are devoted to generalizations of amicable numbers.  
We point out the papers \cite{Beck, Garcia, Hagis1, Hagis2, Lal, Pollack}, but many other similar results can be found in the literature.
Motivated by the aforementioned investigations, in this paper we propose an extension of the notion of amicable numbers.

\section{Lemmas}
\indent

\begin{lemma}\label{Erdos}
For a fixed integer $N>0$, the set of integers $l$ for which $N$ does not divide $\sigma(l)$ has density $0$.
\end{lemma}
\begin{proof}
See (\cite{Erdos1}, Lemma 2).
\end{proof}

\begin{lemma}\label{sigma1}
For each number $x\geq 1$, we have
\begin{equation*}
\sum\limits_{n\leq x}\frac{\sigma(n)}{n}<\zeta(2)x\,.
\end{equation*}
\begin{proof}
Using the well-known identity
\begin{equation}\label{identity}
\frac{\sigma(n)}{n}=\sum_{u|n}\frac{1}{u}
\end{equation}
we write
\begin{align*}
&\sum\limits_{n\leq x}\frac{\sigma(n)}{n}=\sum\limits_{n\leq x}\sum_{u|n}\frac{1}{u}=\sum\limits_{u\leq x}\frac{1}{u}\sum_{n\leq x\atop{u|n}}1
=\sum\limits_{u\leq x}\frac{1}{u}\bigg[\frac{x}{u}\bigg]\leq x\sum\limits_{u\leq x}\frac{1}{u^2}<\zeta(2)x\,.\\
\end{align*}
\end{proof}
\end{lemma}

\begin{lemma}\label{sigmak}
Let $k\geq2$ be integer. For each number $x\geq 1$, we have
\begin{equation*}
\sum\limits_{n\leq x}\frac{\sigma^k(n)}{n^k}<\zeta^k(2)\zeta(2k-1)x\,.
\end{equation*}
\begin{proof}
Using \eqref{identity} and the well-known identity
\begin{equation*}
\textmd{lcm}(n_1,\ldots, n_k)\,\textmd{gcd}\bigg(\frac{m}{n_1},\ldots, \frac{m}{n_k}\bigg)=m\,, 
\end{equation*}
where
\begin{equation*}
m=n_1\cdots n_k
\end{equation*}
we deduce
\begin{align*}
\sum\limits_{n\leq x}\frac{\sigma^k(n)}{n^k}&=\sum\limits_{n\leq x}\sum_{n_1|n}\cdots\sum_{n_k|n}\frac{1}{n_1\cdots n_k}
=\sum_{n_1\leq x}\cdots\sum_{n_k\leq x}\frac{1}{n_1\cdots n_k}\sum_{n\leq x\atop{\textmd{lcm}(n_1, \ldots, n_k})|n}1  \\
&=\sum_{n_1\leq x}\cdots\sum_{n_k\leq x}\frac{1}{n_1\cdots n_k}\bigg[\frac{x}{\textmd{lcm}(n_1, \ldots, n_k)}\bigg]\\
&\leq x\sum_{n_1\leq x}\cdots\sum_{n_k\leq x}\frac{1}{n_1\cdots n_k\,\textmd{lcm}(n_1, \ldots, n_k)}\\
&=x\sum_{n_1\leq x}\cdots\sum_{n_k\leq x}\frac{\textmd{gcd}\left(\frac{m}{n_1},\ldots, \frac{m}{n_k}\right)}{n^2_1\cdots n^2_k}
<x\sum\limits_{n_1=1}^\infty\cdots\sum\limits_{n_k=1}^\infty\frac{\textmd{gcd}\left(\frac{m}{n_1},\ldots, \frac{m}{n_k}\right)}{n^2_1\cdots n^2_k}\\
&=x\sum\limits_{d=1}^\infty\sum\limits_{\textmd{gcd}\left(\frac{m}{n_1},\ldots, \frac{m}{n_k}\right)=d}\frac{d}{n^2_1\cdots n^2_k}
<x\sum\limits_{d=1}^\infty\sum\limits_{n_1=1}^\infty\cdots\sum\limits_{n_k=1}^\infty\frac{d}{(dn'_1)^2\cdots(dn'_k)^2}\\
&=x\sum\limits_{d=1}^\infty\frac{1}{d^{2k-1}}\sum\limits_{n_1=1}^\infty\cdots\sum\limits_{n_k=1}^\infty\frac{1}{(n'_1)^2\cdots(n'_k)^2}=\zeta^k(2)\zeta(2k-1)x\,.
\end{align*}
\end{proof}
\end{lemma}

\section{{\boldmath$(\alpha_1,\ldots,\alpha_k)$}-multiamicable {\boldmath$k$}-tuples}  
\indent

We say that the natural numbers $n_1,\ldots, n_k$ form an $(\alpha_1,\ldots,\alpha_k)$-multiamicable $k$-tuple if
\begin{equation}\label{sigma1k}
\sigma(n_1)=\sigma(n_2)=\cdots=\sigma(n_k)=\alpha_1n_1+\alpha_2n_2+\cdots+\alpha_kn_k\,.
\end{equation}
Of course, $\alpha_1,\ldots,\alpha_k$ must be positive integers. 
If $k=1$ and $\alpha_1=2$ then $n_1$ is a perfect number. 
When $\alpha_1=\alpha_2=\cdots=\alpha_k$ then $n_1,\ldots, n_k$ is a multiply amicable $k$-tuple according to Mason's definition. 
There is no extra interest in allowing the possibility that $n_1=n_2=\cdots =n_k$, so we shall assume always that $n_1<n_2<\cdots<n_k$
with the $k$-tuple $(\alpha_1,\ldots,\alpha_k)$ ordered accordingly. Hence
\begin{equation}\label{abundant}
(\alpha_1+\cdots+\alpha_k)n_1<\sigma(n_j)<(\alpha_1+\cdots+\alpha_k)n_k
\end{equation}
for each $j\in[1, k]$, which means that $n_1$ is $(\alpha_1+\cdots+\alpha_k)$-abundant. 
Following the method in \cite{Dickson} and \cite{Mason} we shell use the next proposition to find $(\alpha_1,\ldots,\alpha_k)$-multiamicable $k$-tuples.
\begin{theorem}
Suppose the natural numbers $N_1,\ldots, N_k$,  $\alpha_1,\ldots,\alpha_k$ and $a$ satisfy $(a, N_1)=\cdots=(a, N_k)= 1$ and
\begin{equation*}
\frac{\sigma(a)}{a}=\frac{\alpha_1N_1+\cdots+\alpha_kN_k}{\sigma(N_1)}=\cdots=\frac{\alpha_1N_1+\cdots+\alpha_kN_k}{\sigma(N_k)}\,.
\end{equation*}
Then $aN_1,\ldots, aN_k$ are an $(\alpha_1,\ldots,\alpha_k)$-multiamicable $k$-tuple.
\end{theorem}

\begin{proof}
This follows  directly from the multiplicativity of $\sigma$.
\end{proof}
For $k=2$ several of these pairs $(aN_1, aN_2)$  that are $(\alpha_1, \alpha_2)$-multiamicable pairs are listed in the table below.
\begin{center}
\begin{tabular}[t]{|p{1em}|p{1em}|p{3.6em}|p{3.8em}|p{3.2em}|p{3.2em}|p{16em}|}
\hline 
$\alpha_1$ & $\alpha_2$ & $N_1$ & $N_2$ & $\sigma(a)/a$ & $a$ &  \hspace{22mm} $(aN_1, aN_2)$  \\
\hline
1 & 2 & $2^313$ & $2^2 29$ & 8/5 & $3\cdot5$ & \hspace{20mm} (1560, 1740) \\
\hline
1 & 2 & $2^2 3^25\cdot41$ & $2^5 3^5$ & 1 & 1 & \hspace{20mm}  (7380, 7776) \\
\hline
1 & 2 & $2^3 41$ & $2^2 89$ & 104/63 & $3^2 7$ & \hspace{18mm}  (20664, 22428) \\
\hline
1 & 2 &  $17\cdot37$ & 683 & 35/12 & $2^5 3^3$  & \hspace{16mm} (543456, 590112) \\
\hline
1 & 2 & $17\cdot37$ & 683 & 35/12 & $2^3 3^2 13$ & \hspace{16mm} (588744, 639288) \\
\hline
1 & 2 & $7^211\cdot17$ & $107\cdot113$ & 65/24 & $2^3 3^2$ & \hspace{16mm} (659736, 870552) \\
\hline
1 & 2 & $13\cdot89$ & $1259$ & 35/12 & $2^5 3^3$ & \hspace{16mm} (999648, 1087776) \\
\hline
2 & 1 & $2^2 5\cdot107$ & $2^5 71$ & 13/9 & $3^2$  & \hspace{18mm} (19260, 20448) \\
\hline
2 & 1 & $2^3 3^5 11$ &  $2^5 3^2 79$ & 1 & 1 & \hspace{18mm} (21384, 22752) \\
\hline
2 & 1 & $2^2 3^3 29$  & $2^3 3\cdot139$  & 8/7 & 7 & \hspace{18mm} (21924, 23352)  \\
\hline
2 & 1 & $17\cdot37\cdot59$ & $179\cdot227$ & 403/144 &  $2^4 3^2$  & \hspace{14mm} (5343984, 5851152) \\
\hline
1 & 3 & $3^3 5^3$ & $3^2 5\cdot79$ & 9/4 & $2^5 7$ & \hspace{16mm} (756000, 796320) \\
\hline
3 & 1 & $11\cdot29$ & $17\cdot19$ & 32/9 & $2^2 3^3 5\cdot7$ & \hspace{14mm} (1205820, 1220940) \\
\hline
3 & 1 & $7\cdot13^2$ & $11^3$ & 10/3 & $2^3 3^3 5$ & \hspace{14mm} (1277640, 1437480) \\
\hline
3 & 1 & $3^2 19\cdot41$ & $3^5 29$ & 18/7 &  $2^3 5\cdot7$ & \hspace{14mm} (1963080, 1973160) \\
\hline
\end{tabular}
\captionof{table}{}\label{Table1}
\end{center}
The sequence \seqnum{A383239} in the OEIS \cite{Sloane} consists of the larger components of $(1, 2)$-multiamicable pairs.
The next result concerns the density of $(\alpha, \beta)$-multiamicable pairs. Our argument is a modification of the one used by Cohen, Gretton, and Hagis in \cite{Cohen}.
\begin{theorem}\label{Mox}
Let $M(x)$ denote the number of $(\alpha, \beta)$-multiamicable pairs $m$, $n$ with $m<n$ and $m\leq x$. Then $M(x)=o(x)$ as $x\rightarrow \infty$.
\end{theorem}

\begin{proof}
According to \eqref{sigma1k} we have   
\begin{equation}\label{sigmamn}
\sigma(m)=\sigma(n)=\alpha m+\beta n\,.
\end{equation}
Assume that $K$ is a large positive  integer.
Let $M_1(x)$ denote the number of $(\alpha, \beta)$-multiamicable pairs $m$, $n$ with $m\leq x$ and $\sigma(m)/m\geq K$.
From \eqref{abundant} we derive
\begin{equation}\label{sigmamab}
\frac{\sigma(m)}{m}>\alpha+\beta\,.
\end{equation}
Therefore the number of $n$'s that can correspond to a given value of $m$ is less than $\frac{1}{2}(\sigma(m)/m-1)(\sigma(m)/m-2)$.
Now Lemma \ref{sigma1} and Lemma \ref{sigmak} give us
\begin{align}\label{M1est}
M_1(x)&\leq\frac{1}{2}\sum\limits_{m\leq x\atop{\sigma(m)/m\geq K}}\bigg(\frac{\sigma(m)}{m}-1\bigg)\bigg(\frac{\sigma(m)}{m}-2\bigg)\nonumber\\
&\leq\frac{1}{2K}\sum\limits_{m\leq x}\frac{\sigma(m)}{m}\bigg(\frac{\sigma(m)}{m}-1\bigg)\bigg(\frac{\sigma(m)}{m}-2\bigg)\nonumber\\
&<\frac{1}{K}\sum\limits_{m\leq x}\frac{\sigma(m)}{m}+\frac{3}{2K}\sum\limits_{m\leq x}\frac{\sigma^2(m)}{m^2}+\frac{1}{2K}\sum\limits_{m\leq x}\frac{\sigma^3(m)}{m^3}\nonumber\\
&<\frac{1}{K}\zeta(2)x+\frac{3}{2K}\zeta^2(2)\zeta(3)x+\frac{1}{2K}\zeta^3(2)\zeta(5)x\,.
\end{align}
Let $M_2(x)$ denote the number of $(\alpha, \beta)$-multiamicable pairs $m$, $n$ with $m \leq x$ and $\sigma(m)/m<K$. By \eqref{sigmamab} we get
\begin{equation*}
\alpha+\beta<K\,.
\end{equation*}
Consequently the number of $n$'s that can correspond to a fixed value of $m$ is less than $\frac{1}{2}(K-1)(K-2)$. 
Using this consideration, \eqref{sigmamn} and Lemma \ref{Erdos} we obtain 
\begin{align*}
M_2(x)&\leq\frac{1}{2}(K-1)(K-2)\sum\limits_{m\leq x\atop{K^4\mid\sigma(m)}}1+\frac{1}{2}(K-1)(K-2)\sum\limits_{m\leq x\atop{K^4\nmid\sigma(m)}}1 \nonumber\\
&\leq\frac{1}{2}(K-1)(K-2)\sum\limits_{m\leq x\atop{\alpha m+\beta n\equiv 0\, ( \textmd{mod}\,K^4)}}1+o(x) \nonumber\\
\end{align*}

\begin{align}\label{M2est}
&\leq\frac{1}{2}(K-1)(K-2)\bigg(\frac{x}{K^3}+\mathcal{O}\big(K\big)\bigg)+o(x) \nonumber\\
&\leq\frac{x}{K}+o(x)\,.
\end{align}
Apparently 
\begin{equation}\label{M12}
M(x)=M_1(x)+M_2(x)\,.
\end{equation}
Summarizing \eqref{M1est}, \eqref{M2est}, \eqref{M12} and bearing in mind that $K$ can be taken arbitrarily large we establish  
\begin{equation*}
M(x)=o(x) \quad \mbox{ as } \quad x\rightarrow \infty.
\end{equation*}
This completes the proof of Theorem \ref{Mox}.
\end{proof}
The following conjecture provides a natural conclusion to this section.
\begin{conjecture}
For any fixed positive integers $\alpha_1,\ldots,\alpha_k$, there exist  infinitely many $(\alpha_1,\ldots,\alpha_k)$-multiamicable $k$-tuples. 
\end{conjecture}

\section{{\boldmath$(\alpha, \beta)$}-amicable pairs}
\indent

Let $\alpha$ and $\beta$ be positive integers. We say that the numbers $m$ and $n$ form an $(\alpha, \beta)$-amicable pair if
\begin{equation*}
\sigma(\alpha n)-\alpha n=m \quad \mbox{ and }   \quad \sigma(\beta m)-\beta m=n
\end{equation*}
or equivalently 
\begin{equation*}
m=s(\alpha n) \quad \mbox{ and }   \quad n=s(\beta m)\,.
\end{equation*}
When $\alpha=\beta=1$ then $m$ and $n$ are amicable. 
Several $(\alpha, \beta)$-amicable pairs are listed in the table below.
\begin{center}
\begin{tabular}[t]{|p{0.5em}|p{0.5em}|p{12em}|}
\hline 
$\alpha$ & $\beta$ &\hspace{3mm} $(\alpha, \beta)$-amicable pairs \\
\hline
1 & 2 &  $(26, 46)$, $(296, 586)$ \\
\hline
1 & 3 &  $(3, 4)$, $(15, 33)$, $(5919, 7905)$ \\
\hline
\end{tabular}
\captionof{table}{}\label{Table2}
\end{center}
The sequence \seqnum{A384411} in the OEIS \cite{Sloane} consists of $(1, 2)$-amicable pairs.
The OEIS \cite{Sloane} sequence \seqnum{A383483} lists the smaller elements of $(1, 3)$-amicable pairs.

\section{{\boldmath$\mathrm{PM(p, q)}$}-amicable {\boldmath$k$}-tuples}
\indent

The abbreviation $\textmd{PM}$ stands for Power Mean. Let $k\geq1$, $p\geq1$  and $q\geq1$ be integers.
We say that the numbers $n_1,\ldots, n_k$ form a $\textmd{PM(p, q)}$--amicable $k$-tuple if
\begin{equation*}
\sigma^p(n_1)+\sigma^p(n_2)+\cdots+\sigma^p(n_k)=q(n_1+n_2+\cdots+n_k)^p\,.
\end{equation*}
Obviously, when $k=q$, then every amicable  $k$-tuple satisfies the above equation.
Several $\textmd{PM(p, q)}$-amicable $k$-tuples that are not amicable  $k$-tuples are listed in the table below.

\begin{center}
\begin{tabular}[t]{|p{0.5em}|p{0.5em}|p{0.5em}|p{22em}|}
\hline 
$k$ & $p$ & $q$ & \hspace{18mm} $\textmd{PM(p, q)}$-amicable $k$-tuples  \\
\hline
2 & 1 & 2 & $(3, 20)$, $(5, 12)$, $(5, 70)$, $(5, 88)$, $(6, 28)$, $(10, 20)$ \\
\hline
2 & 1 & 3 & $(6, 180)$, $(10, 780)$, $(24, 780)$, $(26, 660)$, $(34, 504)$ \\
\hline
2 & 2 & 1 & $(2, 3)$, $(19, 33)$, $(27, 77)$, $(39, 161)$, $(45, 133)$, $(51, 69)$ \\
\hline
2 & 2 & 2 & $(1, 4)$, $(1378, 9962)$,   $(1660, 4892)$, $(1975, 10425)$   \\
\hline
3 & 1 & 2 &  $(1, 2, 20)$, $(1, 3, 18)$, $(1, 4, 20)$,  $(1, 8, 20)$, $(1, 10, 18 )$\\
\hline
3 & 1 & 3 &  $(1, 76, 360)$, $(2, 11, 240)$, $(2, 41, 420)$,  $(6, 120, 180)$ \\
\hline
3 & 2 & 1 &  $(1, 47, 185)$, $(2, 11, 14)$, $(2, 110, 371)$, $(3, 302, 411)$  \\
\hline
3 & 2 & 2 &  $(3, 36, 98)$, $(5, 34, 135)$, $(5, 40, 105)$, $(10, 106, 406)$  \\
\hline
3 & 2 & 3 &  $(14, 350, 1340)$, $(22, 96, 1862)$, $(31, 301, 1876)$   \\
\hline
3 & 2 & 4 &  $(3, 39, 156)$, $(11, 40, 294)$, $(12, 14, 60)$, $(17, 70, 210)$  \\
\hline
3 & 2 & 5 & $(6, 222, 1608)$, $(15, 33, 168)$, $(30, 66, 552)$  \\
\hline
3 & 3 & 1 &  $(2, 10, 15)$, $(4, 20, 39)$, $(8, 40, 87)$, $(9, 45, 63)$ \\
\hline
3 & 3 & 3 & $(56, 134, 710)$, $(108, 268, 1724)$, $(236, 404, 2510)$  \\
\hline
\end{tabular}
\captionof{table}{}\label{Table3}
\end{center}
The sequence \seqnum{A383484} in the OEIS \cite{Sloane} consists of the larger components of $\textmd{PM(2, 1)}$-amicable pairs.
The OEIS \cite{Sloane} sequence \seqnum{A385008} lists the larger elements of $\textmd{PM(2, 2)}$-amicable pairs.

\section{{\boldmath$\mathrm{WPM(p)}$}-amicable {\boldmath$k$}-tuples}
\indent

The abbreviation $\textmd{WPM}$ stands for Weighted Power Mean. Let $k\geq1$ and $p\geq1$ be integers. 
We say that the numbers $n_1,\ldots, n_k$ form an $\textmd{WPM(p)}$-amicable $k$-tuple if
\begin{equation*}
n_1\sigma^p(n_1)+n_2\sigma^p(n_2)+\cdots+n_k\sigma^p(n_k)=(n_1+n_2+\cdots+n_k)^{p+1}\,.
\end{equation*}
Apparently, each amicable $k$-tuple satisfies the above equation.
Several $\textmd{WPM(p)}$-amicable $k$-tuples that are not amicable $k$-tuples are listed in the table below.

\begin{center}
\begin{tabular}[t]{|p{0.5em}|p{0.5em}|p{19em}|}
\hline 
$k$ & $p$ &\hspace{10mm} $\textmd{WPM(p)}$-amicable $k$-tuples  \\
\hline
2 & 1 &  $(4, 6)$, $(10, 16)$, $(34, 68)$, $(60, 81)$, $(91, 273)$  \\
\hline
2 & 2 &  $(7, 21)$, $(105, 231)$, $(1065, 2499)$ \\
\hline
3 & 1 &  $(1, 21, 63)$, $(1, 22, 44)$, $(2, 38, 98)$, $(4, 6, 34)$ \\
\hline
3 & 2 & $(12, 276, 412)$, $(70, 210, 224)$, $(87, 189, 264)$  \\
\hline
\end{tabular}
\captionof{table}{}\label{Table4}
\end{center}
The sequence \seqnum{A383714} in the OEIS \cite{Sloane} consists of the larger components of $\textmd{WPM(2)}$-amicable pairs.

\section{{\boldmath$\mathrm{GM}$}-amicable {\boldmath$k$}-tuples}
\indent

The abbreviation $\textmd{GM}$ stands for Geometric Mean. Let $k\geq1$ be integer. We say that the numbers $n_1,\ldots, n_k$ form a $\textmd{GM}$-amicable $k$-tuple if
\begin{equation*}
\sigma(n_1)\sigma(n_2)\cdots\sigma(n_k)=(n_1+n_2+\cdots+n_k)^k\,.
\end{equation*}
It is easy to see that every amicable  $k$-tuple satisfies the above equation.
Several $\textmd{GM}$-amicable $k$-tuples that are not amicable  $k$-tuples are listed in the table below.

\begin{center}
\begin{tabular}[t]{|p{0.5em}|p{27.2em}|}
\hline 
$k$ & \hspace{33mm}  $\textmd{GM}$-amicable $k$-tuples  \\
\hline
2 & $(28, 84)$, $(42, 102)$, $(60, 276)$, $(92, 160)$, $(244, 624)$, $(426, 582)$ \\
\hline
3 & $(1080, 1092, 1188)$, $(10164, 10584, 11172)$, $(10440, 10692, 11628)$ \\
\hline
\end{tabular}
\captionof{table}{}\label{Table5}
\end{center}
The OEIS \cite{Sloane} sequences \seqnum{A383932} and \seqnum{A386010} list the larger elements of $\textmd{GM}$-amicable pairs and $\textmd{GM}$-amicable triples, respectively.

\section{ {\boldmath$\mathrm{WGM}$}-amicable {\boldmath$k$}-tuples}
\indent

The abbreviation $\textmd{WGM}$ stands for Weighted Geometric Mean. Let $k\geq1$ be integer. We say that the numbers $n_1,\ldots, n_k$ form an $\textmd{WGM}$-amicable $k$-tuple if
\begin{equation*}
\sigma(n_1)^{n_1}\sigma(n_2)^{n_2}\cdots\sigma(n_k)^{n_k}=(n_1+n_2+\cdots+n_k)^{n_1+n_2+\cdots+n_k}\,.
\end{equation*}
It is clear that each amicable  $k$-tuple satisfies the above equation.
Although we cannot set an example, we cannot rule out the existence of $\textmd{WGM}$-amicable $k$-tuples that are not amicable $k$-tuples.
The sequence \seqnum{A385186} in the OEIS \cite{Sloane} consists of the larger components of $\textmd{WGM}$-amicable pairs.

\section{ {\boldmath$\mathrm{IWGM}$}-amicable {\boldmath$k$}-tuples}
\indent

The abbreviation $\textmd{IWGM}$ stands for Inverse Weighted Geometric Mean. The positive integers $n_1,\ldots, n_k$ form an $\textmd{IWGM}$-amicable $k$-tuple if
\begin{equation*}
n_1^{\sigma(n_1)}n_2^{\sigma(n_2)}\cdots n_k^{\sigma(n_k)}=(n_1n_2\cdots n_k)^{n_1+n_2+\cdots+n_k}\,.
\end{equation*}
Apparently, each amicable  $k$-tuple satisfies the above equation. Several $\textmd{IWGM}$-amicable pairs that are not amicable pairs are listed in the table below.
\begin{center}
\begin{tabular}[t]{|p{37em}|}
\hline 
\hspace{54mm} $\textmd{IWGM}$-amicable pairs \\
\hline
 $(1, 2)$, $(1, 3)$, $(1, 5)$, $(1, 7)$, $(1, 11)$, $(1, 13)$, $(1, 17)$, $(1, 19)$, $(1, 23)$, $(1, 29)$, $(1, 31)$, $(1, 37)$  \\     
\hline
\end{tabular}
\captionof{table}{}\label{Table6}
\end{center}
The OEIS \cite{Sloane} sequence \seqnum{A385492} lists the larger elements of $\textmd{IWGM}$-amicable pairs.

\section{{\boldmath$\mathrm{HM(p, q)}$}-amicable {\boldmath$k$}-tuples}
\indent

The abbreviation $\textrm{HM}$ stands for Harmonic Mean. Let $k\geq1$, $p\geq1$  and $q\geq1$ be integers. We say that the numbers $n_1,\ldots, n_k$ form a $\textmd{HM(p, q)}$-amicable $k$-tuple if
\begin{equation*}
\Bigg(\frac{1}{\sigma^p(n_1)}+\frac{1}{\sigma^p(n_2)}+\cdots+\frac{1}{\sigma^p(n_k)}\Bigg)(n_1+n_2+\cdots+n_k)^p=q\,.
\end{equation*}
Obviously, if $k=q$, then every amicable  $k$-tuple satisfies the above equation.
Several $\textrm{HM(p, q)}$-amicable $k$-tuples that are not amicable  $k$-tuples are listed in the table below.

\begin{center}
\begin{tabular}[t]{|p{0.5em}|p{0.5em}|p{0.5em}|p{24em}|}
\hline 
$k$ & $p$ & $q$ & \hspace{22mm} $\textmd{HM(p, q)}$-amicable $k$-tuples  \\
\hline
2 & 1 & 2 & $(20, 28)$, $(24, 56)$, $(30, 66)$, $(40, 90)$, $(56, 88)$, $(92, 132)$ \\
\hline
2 & 1 & 3 & $(3, 6)$, $(10, 32)$, $(12, 60)$, $(15, 33)$, $(24, 116)$, $(33, 57)$ \\
\hline
2 & 2 & 1 & $(120, 168)$, $(1272, 1320)$, $(2160, 3792)$, $(3672, 4968)$ \\
\hline
2 & 2 & 4 & $(435, 717)$, $(447, 513)$, $(2001, 2607)$, $(2001, 2607)$ \\
\hline
3 & 1 & 3 &  $(840, 1020, 1380)$, $(1008, 1260, 1638)$, $(2016, 2232, 2772)$\\
\hline
3 & 2 & 4 &  $(480, 480, 1056)$, $(1400, 1400, 2160)$, $(3936, 3936, 6240)$\\
\hline
\end{tabular}
\captionof{table}{}\label{Table7}
\end{center}
The sequence \seqnum{A384814} in the OEIS \cite{Sloane} consists of the larger components of $\textmd{HM(1, 2)}$-amicable pairs.
The OEIS \cite{Sloane} sequence \seqnum{A383964} lists the larger elements of $\textmd{HM(2, 1)}$-amicable pairs.
The sequence \seqnum{A385155} in the OEIS \cite{Sloane} consists of the larger members of $\textmd{HM(1, 3)}$-amicable triples.

\section{{\boldmath$\mathrm{WHM(p)}$}-amicable {\boldmath$k$}-tuples}\label{WHM}
\indent

The abbreviation $\textmd{WHM}$ stands for Weighted Harmonic Mean. 
Let $k\geq1$ and $p\geq1$ be integers. We say that the numbers $n_1,\ldots, n_k$ form an $\textmd{WHM(p)}$-amicable $k$-tuple if
\begin{equation*}
\Bigg(\frac{n^p_1}{\sigma^p(n_1)}+\frac{n^p_2}{\sigma^p(n_2)}+\cdots+\frac{n^p_k}{\sigma^p(n_k)}\Bigg)(n_1+n_2+\cdots+n_k)^p=n^p_1+n^p_2+\cdots+n^p_k\,.
\end{equation*}
Apparently, each amicable $k$-tuple satisfies the above equation.
It is easy to see that $\textmd{WHM(1)}$-amicable $k$-tuple coincides with feebly amicable $k$-tuple defined in \cite{Bozarth}.
Several $\textmd{WHM(p)}$-amicable $k$-tuples that are not amicable $k$-tuples are listed in the table below.

\begin{center}
\begin{tabular}[t]{|p{0.5em}|p{0.5em}|p{25em}|}
\hline 
$k$ & $p$ &\hspace{24mm} $\textmd{WHM(p)}$-amicable $k$-tuples \\
\hline
3 & 1 & $(72, 360, 504)$, $(84, 120, 840)$, $(84, 672, 840)$, $(96, 660, 840)$   \\    
\hline
3 & 2 & $(117, 117, 4680)$\\    
\hline
\end{tabular}
\captionof{table}{}\label{Table8}
\end{center}
The OEIS \cite{Sloane} sequences \seqnum{A384487} and \seqnum{A385749} list the larger elements of $\textmd{WHM(1)}$- and  $\textmd{WHM(2)}$-amicable triples, respectively.

\section{Cross-Harmonious pairs}
\indent

Let $a$ and $b$ be positive integers. We say that $a$ and $b$ form a cross-harmonious pair if 
\begin{equation*}
\frac{b}{\sigma(a)}+\frac{a}{\sigma(b)}=1\,.
\end{equation*}
Clearly, every amicable pair is cross-harmonious. Several cross-harmonious pairs that are not amicable pairs are listed in the table below.
\begin{center}
\begin{tabular}[t]{|p{35em}|}
\hline 
\hspace{52mm} Cross-harmonious pairs \\
\hline
 $(12, 14)$, $(12, 20)$, $(30, 54)$, $(32, 42)$, $(48, 62)$, $(70, 88)$, $(70, 108)$, $(78, 126)$, $(88, 114)$  \\     
\hline
\end{tabular}
\captionof{table}{}\label{Table9}
\end{center}
The sequence \seqnum{A384706} in the OEIS \cite{Sloane} consists of the larger components of cross-harmonious pairs.

\section{{\boldmath$\mathrm{MP(p, q)}$}-amicable {\boldmath$k$}-tuples}
\indent

The abbreviation $\textmd{MP}$ stands for Mean Power. Let $k\geq1$, $p\geq2$  and $q\geq1$ be integers. We say that the numbers $n_1,\ldots, n_k$ form a $\textmd{MP(p, q)}$-amicable $k$-tuple if
\begin{equation*}
\sigma^p(n_1)+\sigma^p(n_2)+\cdots+\sigma^p(n_k)=q(n^p_1+n^p_2+\cdots+n^p_k)\,.
\end{equation*}
Several $\textmd{MP(p, q)}$-amicable $k$-tuples are listed in the table below.

\begin{center}
\begin{tabular}[t]{|p{0.5em}|p{0.5em}|p{0.5em}|p{30.6em}|}
\hline 
$k$ & $p$ & $q$ & \hspace{34mm} $\textmd{MP(p, q)}$-amicable $k$-tuples  \\
\hline
2 & 2 & 2 & $(1, 2)$, $(13, 21)$, $(13, 27)$, $(17, 175)$, $(45, 123)$, $(1069, 2133)$, $(1093, 2187)$ \\
\hline
2 & 2 & 4 & $ (6, 28)$, $(12, 14)$, $(48, 62)$, $(112, 124)$, $(135, 208)$, $(160, 189)$, $(192, 254)$ \\
\hline
3 & 2 & 2 & $(2, 4, 51)$, $(3, 40, 71)$, $(5, 12, 23)$, $(7, 116, 303)$, $(11, 20, 55)$, $(12, 215, 333)$\\
\hline
3 & 2 & 3 & $(1, 81, 148)$, $(10, 94, 164)$, $(14, 20, 34)$, $(20, 82, 116)$, $(26, 70, 418)$ \\
\hline
4 & 3 & 3 & $(1, 4, 5, 9)$, $(32, 34, 202, 245)$, $(52, 108, 135, 233)$, $(55, 58, 230, 281)$ \\
\hline
4 & 3 & 4 & $(2, 49, 56, 118)$, $(7, 35, 51, 75)$, $ (10, 168, 207, 307)$, $(37, 74, 232, 253)$\\
\hline
4 & 3 & 5 & $(11, 25, 95, 148)$, $(15, 59, 128, 129)$, $(35, 170, 186, 237)$, $(44, 125, 139, 266)$ \\
\hline
4 & 3 & 6 & $(6, 80, 85, 135)$, $(20, 22, 34, 92)$, $(41, 66, 123, 190)$, $(41, 70, 107, 190)$ \\
\hline
4 & 3 & 7 & $(3, 181, 198, 212)$, $(11, 21, 24, 25)$, $(14, 30, 94, 102)$, $(23, 81, 112, 135)$ \\
\hline
4 & 3 & 8 & $(13, 91, 116, 176)$, $(30, 74, 82, 112)$, $(48, 62, 112, 124)$, $(48, 62, 160, 189)$ \\
\hline
4 & 3 & 9 & $(18, 55, 178, 220)$, $(28, 231, 273, 320)$, $(36, 43, 44, 78)$, $(44, 46, 96, 258)$ \\
\hline
\end{tabular}
\captionof{table}{}\label{Table10}
\end{center}
The OEIS \cite{Sloane} sequence \seqnum{A384255} lists the larger members of $\textmd{MP(2,2)}$-amicable pairs.
\begin{question}
Is there an $MP(2, 2)$-amicable pair $(m, n)$ such that $\sigma(m)=\sigma(n)$?
More precisely, we raise the question of the existence of a number pair $(m, n)$ such that
\begin{equation*}
\sigma^2(m)=\sigma^2(n)=m^2+n^2\,.
\end{equation*}
\end{question}

\section{Sigma-Quadratic triples}
\indent

The positive integers $a\leq b$ and $c$ form a sigma-quadratic triple if 
\begin{equation*}
\sigma^2(a)=\sigma^2(b)=a^2+b^2+c^2\,.
\end{equation*}
Several sigma-quadratic triples are listed in the table below.
\begin{center}
\begin{tabular}[t]{|p{37em}|}
\hline 
\hspace{52mm} Sigma-Quadratic triples \\
\hline
$(2, 2, 1)$, $(40, 40, 70)$, $(40, 58, 56)$, $(164, 194, 148)$, $(196, 196, 287)$, $(224, 224, 392)$, $(1120, 1120, 2576)$, 
$(3040, 4024, 5632)$, $(44200, 49300, 96680)$, $(687184, 703312, 965780)$ \\     
\hline
\end{tabular}
\captionof{table}{}\label{Table11}
\end{center}
The sequence \seqnum{A385356} in the OEIS \cite{Sloane} consists of the first components of sigma-quadratic triples.

\section{Sigma-Quadratic quadruples}
\indent

The positive integers $a\leq b\leq c$ and $d$ form a sigma-quadratic quadruple if 
\begin{equation*}
\sigma^2(a)=\sigma^2(b)=\sigma^2(c)=a^2+b^2+c^2+d^2\,.
\end{equation*}
Several sigma-quadratic quadruples are listed in the table below.
\begin{center}
\begin{tabular}[t]{|p{35em}|}
\hline 
\hspace{50mm} Sigma-Quadratic quadruples \\
\hline
$(4, 4, 4, 1)$, $(204, 220, 220, 340)$, $(3480, 3672, 4296, 8520)$, $(4796, 4988, 5276, 3108)$, $(5532, 6812, 6812, 6628)$, $(7152, 9272, 9584, 10816)$, $(33468, 35380, 43676, 42700)$\\     
\hline
\end{tabular}
\captionof{table}{}\label{Table12}
\end{center}

The OEIS \cite{Sloane} sequence \seqnum{A385531} lists the first elements of sigma-quadratic quadruples.

\section{Sigma-Cubic triples}
\indent

The positive integers $a$, $b$, and $c$ form a sigma-cubic triple if 
\begin{equation*}
\sigma^3(a)=a^3+b^3+c^3\,.
\end{equation*}
Several sigma-cubic triples are listed in the table below.
\begin{center}
\begin{tabular}[t]{|p{32em}|}
\hline 
\hspace{50mm} Sigma-Cubic triples \\
\hline
$(5, 3, 4)$, $(6, 8, 10)$, $(53, 12, 19)$, $(58, 59, 69)$, $(102, 26, 208)$, $(102, 117, 195)$, $(118, 116, 138)$, $(152, 172, 264)$, $(168, 336, 408)$, $(197, 27, 46)$, $(214, 173, 267)$ \\     
\hline
\end{tabular}
\captionof{table}{}\label{Table13}
\end{center}
The sequence \seqnum{A385325} in the OEIS \cite{Sloane} consists of the first members of sigma-cubic triples.

\section{Sigma-Cubic quadruples}
\indent

The positive integers $a\leq b$, $c$ and $d$ form a sigma-cubic quadruple if 
\begin{equation*}
\sigma^3(a)=\sigma^3(b)=a^3+b^3+c^3+d^3\,.
\end{equation*}
Several sigma-cubic quadruples are listed in the table below.
\begin{center}
\begin{tabular}[t]{|p{27em}|}
\hline 
\hspace{35mm} Sigma-Cubic quadruples \\
\hline
$(153, 153, 105, 165)$, $(216, 216, 168, 576)$, $(255, 321, 84, 312)$, $(324, 324, 271, 804)$, $(672, 910, 147, 1925)$, $(735, 1243, 215, 615)$ \\     
\hline
\end{tabular}
\captionof{table}{}\label{Table14}
\end{center}
The OEIS \cite{Sloane} sequence \seqnum{A385397} lists the first  components of sigma-cubic quadruples.

\section{Sigma-Cubic quintuples}
\indent

The positive integers $a\leq b\leq c$ and $d\leq e$ form a sigma-cubic quintuple if 
\begin{equation*}
\sigma^3(a)=\sigma^3(b)=\sigma^3(c)=a^3+b^3+c^3+d^3+e^3\,.
\end{equation*}
Several sigma-cubic quintuples are listed in the table below.
\begin{center}
\begin{tabular}[t]{|p{35em}|}
\hline 
\hspace{50mm} Sigma-Cubic quintuples \\
\hline
$(30, 55, 55, 11, 23)$, $(62, 62, 69, 4, 43)$, $(90, 153, 153, 135, 135)$, $(174, 190, 323, 5, 94)$, $(238, 321, 321, 77, 81)$, $(357, 385, 385, 233, 266)$, $(390, 476, 598, 470, 814)$\\     
\hline
\end{tabular}
\captionof{table}{}\label{Table15}
\end{center}
The OEIS \cite{Sloane} sequence \seqnum{A386378} consists of the first elements of sigma-cubic quintuples.

\section{\boldmath$\mathrm{P(i, j)}$-amicable pairs}
\indent

The abbreviation $\textmd{P}$ stands for Prime. We say that the prime numbers $p$ and $q$, where $p<q$, form a $\textmd{P(i, j)}$-amicable pair if
\begin{equation*}
\sigma(p+i)=\sigma(q+j)=p+q\,.
\end{equation*}
It is easy to see that when $i+j=0$ and $(p, q)$ as a $\textmd{P(i, j)}$-amicable pair then $(p+i, q+j)$ is an amicable pair. 
Several  $\textmd{P(i, j)}$-amicable pairs are listed in the table below.
\begin{center}
\begin{tabular}[t]{|p{0.9em}|p{0.9em}|p{34.5em}|}
\hline 
$i$ & $j$ &\hspace{48mm} $\textmd{P(i, j)}$-amicable pairs \\
\hline
1 & 1 &  $(23, 37)$, $(34673, 34687)$, $(55373, 65587)$, $(2056961, 2089951)$, $(5174363, 8161477)$  \\    
\hline
-1 & 1 & $(5021, 5563)$, $(185369, 203431)$, $(308621, 389923)$, $(879713, 901423)$\\    
\hline
-1 & -1 & $(853, 1163)$, $(4513, 7583)$, $(9109, 17099)$, $(44917, 65963)$, $(46183, 48857)$\\    
\hline
1 & 2 & $(179, 367)$, $(263, 457)$, $(557, 691)$, $(332273, 341647)$, $(401309, 909091)$\\    
\hline
1 & 3 & $(19, 23)$, $(29, 43)$, $(137, 151)$, $(223, 281)$, $(727, 953)$, $(22651, 22709)$ \\    
\hline
3 & 1 & $(11, 13)$, $(41, 43)$, $(107, 109)$, $(149, 151)$, $(257, 331)$, $(881, 883)$, $(2141, 2143)$ \\    
\hline
3 & 3 & $(2377, 3671)$, $(7069, 8807)$, $(7949, 9907)$, $(44701, 52067)$, $(71761, 85487)$ \\    
\hline
\end{tabular}
\captionof{table}{}\label{Table16}
\end{center}
The sequences \seqnum{A385586}, \seqnum{A385718}, \seqnum{A385739} and \seqnum{A385740} in the OEIS \cite{Sloane}  list the larger elements of 
$\textmd{P(1, 1)}$-, $\textmd{P(1, 2)}$-, $\textmd{P(-1, 1)}$- and $\textmd{P(-1, -1)}$-amicable pairs, respectively.

\section{Sigma-Quartic quintuples}
\indent

The positive integers $a\leq b$ and $c\leq d\leq e$ form a sigma-quartic quintuple if 
\begin{equation*}
\sigma^4(a)=\sigma^4(b)=a^4+b^4+c^4+d^4+e^4\,.
\end{equation*}
Several sigma-quartic quintuples are listed in the table below.
\begin{center}
\begin{tabular}[t]{|p{30.5em}|}
\hline 
\hspace{40mm} Sigma-Quartic quintuples \\
\hline
$(24, 24, 36, 48, 48)$, $(240, 240, 240, 408, 720)$, $(600, 600, 600, 1020, 1800)$, $(91963648, 91963648, 137945472, 183927296, 183927296)$\\     
\hline
\end{tabular}
\captionof{table}{}\label{Table17}
\end{center}
The OEIS \cite{Sloane} sequence \seqnum{A386225} consists of the first elements of sigma-quartic quintuples.
\begin{remark}
This section does not appear in the published version of the paper in Journal of Integer Sequences. It is included here in the arXiv version for completeness.
\end{remark}

\section{Sigma-Quintic quintuples}
\indent

The positive integers $a$, $b$, $c$, $d$ and $e$ form a sigma-quintic quintuple if 
\begin{equation*}
\sigma^5(a)=a^5+b^5+c^5+d^5+e^5\,.
\end{equation*}
Several sigma-quintic quintuples are listed in the table below.
\begin{center}
\begin{tabular}[t]{|p{29em}|}
\hline 
\hspace{40mm} Sigma-Quintic quintuples \\
\hline
$(46, 19, 43, 47, 67)$, $(94, 38, 86, 92, 134)$, $(946, 418, 1012, 1034, 1474)$, $(1139, 323, 731, 782, 799)$, $(63840, 144480, 154560, 157920, 225120)$ \\     
\hline
\end{tabular}
\captionof{table}{}\label{Table18}
\end{center}
The OEIS \cite{Sloane} sequence \seqnum{A386672} lists the first elements of sigma-quintic quintuples.
\begin{remark}
This section does not appear in the published version of the paper in Journal of Integer Sequences. It is included here in the arXiv version for completeness.
\end{remark}

\vskip20pt
\footnotesize
\begin{flushleft}
S. I. Dimitrov\\
\quad\\
Faculty of Applied Mathematics and Informatics\\
Technical University of Sofia \\
Blvd. St.Kliment Ohridski 8 \\
Sofia 1756, Bulgaria\\
e-mail: sdimitrov@tu-sofia.bg\\
\end{flushleft}

\begin{flushleft}
Department of Bioinformatics and Mathematical Modelling\\
Institute of Biophysics and Biomedical Engineering\\
Bulgarian Academy of Sciences\\
Acad. G. Bonchev Str. Bl. 105, Sofia 1113, Bulgaria \\
e-mail: xyzstoyan@gmail.com\\
\end{flushleft}

\end{document}